\documentclass[12pt]{article}
\usepackage{graphicx, amssymb, amsthm, amsfonts, latexsym, epsfig,
	amsmath, amscd, amsbsy}
\usepackage[usenames]{color}
\usepackage{tikz}
\usepackage{enumerate}
\usepackage{epstopdf}

\usetikzlibrary{decorations.pathmorphing}

\theoremstyle{plain}
\newtheorem{theorem}{Theorem}[section]
\newtheorem{lem}[theorem]{Lemma}

\newtheorem{prop}[theorem]{Proposition}
\theoremstyle{definition}
\newtheorem{df}[theorem]{Definition}
\newtheorem{rem}[theorem]{Remark}

\def\R{{\mathbb R}}
\def\N{{\mathbb N}}
\def\Z{{\mathbb Z}}
\def\C{{\mathcal C}}
\def\eps{\varepsilon}
\renewcommand{\phi}{\varphi}
\def\Q{\mathfrak{Q}}

\def\Int{\mathop\mathrm{Int}}
\def\Cl{\mathop\mathrm{Cl}}

\renewcommand{\setminus}{\smallsetminus}
\def\nin{\notin}
\renewcommand{\hat}{\widehat}
\renewcommand{\tilde}{\widetilde}

\def\cd{\raisebox{2pt}{$\centerdot$}}
\def\rpinf{+^{\hspace{-.1cm}\infty}}

\newcommand{\vecv}{\mathbf{v}}
\newcommand{\vecu}{\mathbf{u}}
\newcommand{\vecw}{\mathbf{w}}
\newcommand{\veczero}{\mathbf{0}}

\begin{document}

\title{Lozi-like maps}
\author{M.\ Misiurewicz
\thanks{This work was partially supported by a grant number 426602
  from the Simons Foundation to Micha{\l} Misiurewicz.}\ \
and
S.\ \v{S}timac
\thanks{Supported in part by the NEWFELPRO Grant No.~24 HeLoMa, and in
  part by the Croatian Science Foundation grant IP-2014-09-2285.}}
\date{}

\maketitle

\begin{abstract}
We define a broad class of piecewise smooth plane homeomorphisms which
have properties similar to the properties of Lozi maps, including the
existence of a hyperbolic attractor. We call those maps
\emph{Lozi-like}. For those maps one can apply our previous results on
kneading theory for Lozi maps. We show a strong numerical evidence
that there exist Lozi-like maps that have kneading sequences different
than those of Lozi maps.
\end{abstract}



\section{Introduction}\label{sec:intro}

This paper can be considered the second part of our paper~\cite{MS}.
In~\cite{MS} we developed three equivalent approaches to compressing
information about the symbolic dynamics of Lozi maps. Let us recall
that Lozi maps are maps of the Euclidean plane to itself, given by the
formula
\begin{equation}\label{e:lozi}
L_{a,b}(x,y)=(1+y-a|x|,b).
\end{equation}
For a large set of parameters $a,b$ this map has a hyperbolic
attractor (see~\cite{M}).

In~\cite{MS} we took a geometric approach, avoiding explicit use of
the piecewise linear formula~\eqref{e:lozi} (which was the base of
elegant results of Ishii~\cite{I}). In fact, we mentioned there that
our aim was to produce a theory that could be applied to a much larger
family of maps. Here we define axiomatically such large family, and
call its members \emph{Lozi-like maps}. We also give a concrete
example of its three-parameter subfamily, containing the family of
Lozi maps.

As a byproduct of considering this subfamily, we extend a little,
compared to~\cite{M}, the region in the parameter plane for the Lozi
maps, for which we can prove that a hyperbolic attractor exists.

Basic characterization of a Lozi map, one of the three obtained
in~\cite{MS}, is the set of \emph{kneading sequences}. While the
situation may appear similar to what we see in one dimension for
unimodal maps, there is a big difference. The one-dimensional analogue
of the Lozi family is the family of tent maps. There we have one
parameter and one kneading sequence. For the Lozi family we have two
parameters, but infinitely many kneading sequences. Thus, by using a
concrete formula~\eqref{e:lozi}, we immensely restrict the possible
sets of kneading sequences. It makes sense to conjecture that in a
generic $n$-parameter subfamily of Lozi-like maps, $n$ kneading
sequences determine all other kneading sequences, at least locally. In
fact, in our example at the end of this paper, we see that for the
Lozi family two kneading sequences may determine the parameter values,
and thus all kneading sequences.

As we mentioned, we wrote~\cite{MS} thinking about possible
generalizations. Therefore for the Lozi-like maps (under suitable
assumptions) all results of that paper hold, and the proofs are the
same, subject only to obvious modification of terminology. There are
only two exceptions, where in~\cite{MS} we used the results
of~\cite{I}. For those exceptions we provide new, general proofs in
the section about symbolic dynamics.

The paper is organized as follows. In Section~\ref{sec:def} we provide
the definitions and prove the basic properties of Lozi-like maps. In
Section~\ref{sec:attractor} we prove the existence of an attractor. In
Section~\ref{sec:symbolic} we give two proofs about symbolic dynamics,
that are different than in~\cite{MS}. Finally, in Section~\ref{sec:ex}
we present an example of a three-parameter family of Lozi-like maps.
We also show that it is essentially larger than the Lozi family. This
last result uses a computer in a not completely rigorous way, so
strictly speaking it is not a proof, but a strong numerical evidence.

\section{Definitions}\label{sec:def}

A \emph{cone} $K$ in $\R^2$ is a set given by a unit vector $\vecv$
and a number $\ell\in(0,1)$ by
\begin{equation}\label{econe1}
K=\{\vecu\in\R^2:|\vecu\cdot\vecv|\ge \ell\|\vecu\|\},
\end{equation}
where $||\cdot||$ denotes the usual Euclidean norm. The straight line
$\{t\vecv:t\in\R\}$ is the \emph{axis} of the cone. Two cones will be
called \emph{disjoint} if their intersection consists only of the
vector $\veczero$. Clearly, the image under an invertible linear
transformation of $\R^2$ of a cone is a cone, although the image of
the axis is not necessarily the axis of the image.

\begin{lem}\label{cone1}
Let $K^u$ and $K^s$ be disjoint cones. Then there is an invertible
linear transformation $T:\R^2\to\R^2$ such that the $x$-axis is the
axis of $T(K^u)$ and the $y$-axis is the axis of $T(K^s)$.
\end{lem}

\begin{proof}
We will define $T$ as the composition of three linear transformations.

Choose two of the four components of $\R^2\setminus(K^u\cup K^s)$,
such that they are not images of each other under the central symmetry
with respect to the origin. Then choose one vector in each of those
components. There is an invertible linear transformation $T_1$ such
that the images of those vectors under $T_1$ are the two basic vectors
$\langle 1,0\rangle$ and $\langle 0,1\rangle$. Then one of the cones
$T_1(K^u)$, $T_1(K^s)$, lies in the first and third quadrants, while
the other one lies in the second and fourth quadrants.

The second transformation, $T_2$, will be given by a matrix of the
form
\[
\begin{bmatrix}
1 & 0 \\
0 & c
\end{bmatrix},
\]
where $c>0$. The cones $T_2(T_1(K^u))$ and $T_2(T_1(K^s))$ vary
continuously with $c$, so their axes also vary continuously with $c$.
We will measure the angle between those axes as the angle between
their halves contained in the first and fourth quadrants. As $c$ goes
to 0, then this angle approaches 0; as $c$ goes to $\infty$, then this
angle approaches $\pi$. Therefore there is a value of $c$ for which
this angle is equal to $\pi/2$. We take this value of $c$ for our
$T_2$.

Now the third transformation, $T_3$, is the rotation that makes the
axis of the cone $T_3(T_2(T_1(K^u)))$ horizontal. Since the axes of
the cones $T_2(T_1(K^u))$ and $T_2(T_1(K^s))$ are perpendicular, so
are the axes of $T_3(T_2(T_1(K^u)))$ and $T_3(T_2(T_1(K^s)))$, and
therefore the axis of $T_3(T_2(T_1(K^s)))$ is vertical. Hence, the
transformation $T=T_3\circ T_2\circ T_1$ satisfies the conditions from
the statement of the lemma.
\end{proof}

Let us note that the above lemma can be also proved using projective
geometry and cross-ratios.

When we say ``smooth'', we will mean ``of class $C^1$''.

\begin{lem}\label{cone2}
If a cone $K$ is given by~\eqref{econe1} with $\vecv=\langle 1,0
\rangle$ (respectively $\vecv=\langle 0,1 \rangle$) and $\gamma$ is a
smooth curve with the tangent vector at each point contained in $K$,
then $\gamma$ is a graph of a function $y=f(x)$ (respectively
$x=f(y)$), which is Lipschitz continuous with constant
$\sqrt{1-\ell^2}/\ell$.
\end{lem}

\begin{proof}
Consider the case $\vecv=\langle 1,0 \rangle$; the other one is
analogous. If $\gamma$ is not a graph of such function, then there are
two points in $\gamma$ with the same $x$-coordinate, so between them
there is a point of $\gamma$ at which the tangent vector is vertical.
However, vertical vectors do not belong to $K$. Thus, $\gamma$ is a
graph of $y=f(x)$, where $f$ is a continuous function. If $x_1\ne x_2$
belong to the domain of $f$, then between $x_1$ and $x_2$ there is a
point $z$ at which the vector tangent to $\gamma$ is parallel to the
vector $\langle x_2-x_1,f(x_2)-f(x_1)\rangle$. By the assumption, this
vector belongs to $K$, so
\[
|x_2-x_1|\ge\ell\sqrt{(x_2-x_1)^2+(f(x_2)-f(x_1))^2}.
\]
This inequality is equivalent to
\[
|f(x_2)-f(x_1)|\le\frac{\sqrt{1-\ell^2}}{\ell}|x_2-x_1|,
\]
so $f$ is Lipschitz continuous with constant $\sqrt{1-\ell^2}/\ell$.
\end{proof}

We will say that a curve as above is \emph{infinite in both
  directions} if the domain of the corresponding function $f$ is all
of $\R$.

\begin{df}\label{df:uc}
We call a pair of cones $K^u$ and $K^s$ \emph{universal} if they are
disjoint and the axis of $K^u$ is the $x$-axis, and the axis of $K^s$
is the $y$-axis.
\end{df}

Note that by Lemma~\ref{cone1}, any two disjoint cones can become
universal via an invertible linear transformation. In fact, we can use
one more linear transformation to make the constants $\ell_u$ and
$\ell_s$ (of universal cones $K^u$ and $K^s$ respectively) equal, but
this does not give us any additional advantage.

\begin{lem}\label{cone3}
Let $K^u$ and $K^s$ be a universal pair of cones. Let $\gamma_u$ and
$\gamma_s$ be smooth curves, infinite in both directions, and with the
tangent vector at each point contained in $K^u$ and $K^s$,
respectively. Then $\gamma_u$ and $\gamma_s$ intersect at exactly one
point.
\end{lem}

\begin{proof}
If $K^u$ and $K^s$ had a common boundary line, then, by the
Pythagorean theorem, we would have $\ell_u^2+\ell_s^2=1$. Since they
are disjoint, those $\ell$s are larger, so $\ell_u^2+\ell_s^2>1$.

By Lemma~\ref{cone2}, $\gamma_u$ is the graph of a function
$y=f_u(x)$, and $\gamma_s$ is the graph of a function $x=f_s(y)$. By
the assumptions, both $f_u$ and $f_s$ are defined on all of $\R$.
Thus, in order to find points of intersection of $\gamma_u$ and
$\gamma_s$, we have to solve the system of equations $y=f_u(x)$,
$x=f_s(y)$. That is, we have to solve the equation $x=f_s(f_u(x))$.

By Lemma~\ref{cone2}, functions $f_u$ and $f_s$ are Lipschitz
continuous with constants $\sqrt{1-\ell_u^2}/\ell_u$ and
$\sqrt{1-\ell_s^2}/\ell_s$, respectively. Therefore the function
$f_s\circ f_u$ is Lipschitz continuous with the constant
\[
\ell=\frac{\sqrt{1-\ell_u^2}\sqrt{1-\ell_s^2}}{\ell_u\ell_s}.
\]
However, the inequality $\ell_u^2+\ell_s^2>1$ (which, as we noticed,
holds) is equivalent to $\ell<1$, so the map $f_s\circ f_u$ is a
contraction. Therefore it has a unique fixed point, which means that
$\gamma_u$ and $\gamma_s$ intersect at exactly one point.
\end{proof}

For an open set $U \subseteq \R^2$, a \emph{cone-field} $\C$ on $U$ is
the assignment of a cone $K_P$ to each point $P \in U$ such that the
axis $\vecv(P)$ and the coefficient $\ell(P)$ vary continuously with
$P$.

\begin{df}\label{df:sh}
Let $F_1, F_2 : \R^2 \to \R^2$ be $C^1$ diffeomorphisms. We say that
$F_1$ and $F_2$ are \emph{synchronously hyperbolic} if they are either
both order reversing, or both order preserving, and there exist
$\lambda > 1$, a universal pair of cones $K^u$ and $K^s$, and cone
fields $\C^u$ and $\C^s$ (consisting of cones $K^u_P$ and $K^s_P$, $P
\in \R^2$, respectively) which satisfy the following properties:
\begin{enumerate}
\item[(S1)] For every point $P \in \R^2$ we have $K^u_P \subset K^u$,
  $K^s_P \subset K^s$, ${DF_i}_P(K^u_P) \subset K^u_{F_i(P)}$, and
  $D{{F_i}_P^{-1}}(K^s_P) \subset K^s_{F_i^{-1}(P)}$, for $i = 1,2$.
\item[(S2)] For every point $P \in \R^2$ and $i = 1,2$ we have
  $||DF_i(\vecu)|| \ge \lambda||\vecu||$ for every $\vecu \in K^u_P$
  and $||DF^{-1}_i(\vecw)|| \ge \lambda||\vecw||$ for every $\vecw \in
  K^s_P$.
\item[(S3)] There exists a smooth curve $\Gamma \subset \R^2$ such
  that for every $P \in \Gamma$ we have $F_1(P) = F_2(P)$, the vector
  tangent to $\Gamma$ at $P$ belongs to $K^s_P$, and the vector
  tangent to $F_i(\Gamma)$ at $F_i(P)$ belongs to $K^u_{F_i(P)}$. We
  require that $\Gamma$ is infinite in both directions.
\end{enumerate}
We call $\Gamma$ the \emph{divider}. It divides the plane into two
parts which we call the \emph{left half-plane} and the \emph{right
  half-plane}. Also $F_1(\Gamma) = F_2(\Gamma)$ divides the plane into
two parts which we call the \emph{upper half-plane} and the
\emph{lower half-plane}.
\end{df}

\begin{rem}\label{newrem}
Since $F_1$ and $F_2$ are either both order reversing, or both order
preserving, for any $P \in \R^2$, $F_1(P)$ and $F_2(P)$ belong to the
same (upper or lower) half-plane. Without loss of generality we may
assume that $F_i$, $i = 1,2$, maps the left half-plane onto the lower
one and the right half-plane onto the upper one.
\end{rem}

Since the existence of the invariant cone fields implies
hyperbolicity (see~\cite[Proposition~5.4.3] {BS}), if $F_1$ and
$F_2$ are synchronously hyperbolic then both are hyperbolic (with
stable and unstable directions of dimension 1). Also, for each of
them, by (S1) and Lemma~\ref{cone2} the stable and unstable manifolds
of any point are infinite in both directions.

Recall that for a map $F$ a \emph{trapping region} is a nonempty set
that is mapped with its closure into its interior. A set $A$ is an
\emph{attractor} if it has a neighborhood $U$ which is a trapping
region, $A=\bigcap_{n=0}^\infty F^n(U)$, and $F$ restricted to $A$ is
topologically transitive.

\begin{df}\label{df:ll}
Let $F_1, F_2 : \R^2 \to \R^2$ be synchronously hyperbolic $C^1$
diffeomorphisms with the divider $\Gamma$. Let $F : \R^2 \to \R^2$ be
defined by the formula
\[
F(P) =
\begin{cases}
F_1(P), & \textrm{if $P$ is in the left  half-plane,}\\
F_2(P), & \textrm{if $P$ is in the right half-plane.}
\end{cases}
\]
We call the map $F$ \emph{Lozi-like} if the following hold:
\begin{enumerate}
\item[(L1)] $-1 < \det DF_i(P) < 0$ for every point $P \in \R^2$ and
  $i = 1, 2$.
\item[(L2)] There exists a trapping region $\Delta$ (for the map $F$),
  which is homeomorphic to an open disk and its closure is
  homeomorphic to a closed disk.
\end{enumerate}
\end{df}

Observe that by Remark~\ref{newrem}, $F$ is a homeomorphism of $\R^2$
onto itself.

Obviously, the Lozi maps $L_{a, b}(x, y) = (1 + y - a|x|, bx)$, with
$a$, $b$ as in~\cite{M}\footnote{$b > 0$, $a\sqrt{2} > b+2$, $b <
  \frac{a^2-1}{2a+1}$, $2a + b < 4$}, provide an example of Lozi-like
maps with $y$-axis as the divider. For the set $\Delta$ we take a
neighborhood of the triangle as in~\cite{M} that usually serves as the
trapping region for the Lozi map.

Let us show some properties of a Lozi-like map $F$ which follow from
the definition.

\begin{lem}\label{lem:fixed point}
\begin{enumerate}
\item[$(P1)$] $\Delta \cap \Gamma \ne \emptyset$ and consequently
  $\Delta \cap F(\Gamma) \ne \emptyset$.
\item[$(P2)$] There exists a unique fixed point in $\Delta$.
\end{enumerate}
Let us denote this fixed point $X$. We may assume without loss of
generality that it belongs to the right half-plane.
\begin{enumerate}
\item[$(P3)$] Let $\lambda_1$ and $\lambda_2$ denote the eigenvalues
  of $DF(X)$ with $|\lambda_1| > 1$ and $|\lambda_2| < 1$. Then
  $\lambda_1 < -1$ and $\lambda_2 > 0$.
\item[$(P4)$] Let $\hat \lambda_1$ and $\hat \lambda_2$ denote the
  eigenvalues of $DF(P) = DF_1(P)$ when $P$ is in the left half-plane,
  and $|\hat \lambda_1| > 1$, $|\hat \lambda_2| < 1$. Then $\hat
  \lambda_1 > 1$ and $\hat \lambda_2 < 0$.
\item[$(P5)$] $X \ne \Gamma \cap F(\Gamma)$ and consequently $X \nin
  \Gamma \cup F(\Gamma)$.
\end{enumerate}
\end{lem}

\begin{proof}
(P1) If $\Delta$ does not intersect the divider, then $\Delta$ is a
  trapping region for $F_i$ for some $i \in \{ 1, 2 \}$, what is not
  possible since $F_i$ is hyperbolic with the unstable direction of
  dimension 1.

(P2) Existence of a fixed point: By the definition of a trapping
  region $F(\Cl \Delta) \subset \Delta$. By (L2) $\Cl \Delta$ is
  homeomorphic to a closed disk and existence of at least one fixed
  point follows by Brouwer's fixed point theorem.

We will prove uniqueness later.

(P3) Let $W^u(X)$ denote the unstable manifold of $F_2$ at $X$.
By~(S3),~(S1) and Lemma~\ref{cone3}, $\Gamma$ and $W^u(X)$ intersect
at exactly one point. Let us denote by $W^u(X)^+$ that half of
$W^u(X)$ which starts at $X$ and completely lies in the right
half-plane. Then for every point $P \in W^u(X)^+$, $F(P) = F_2(P)$.

By (L1) one eigenvalue is positive and the other one is negative. Let
us suppose, by contradiction, that $\lambda_1 > 1$. Then for $P\in
W^u(X)^+$, $F(P)\in W^u(X)^+$. Therefore, for every point $P \in
W^u(X)^+$, $P \ne X$, the distance between $F^n(P)$ and $X$ goes to
infinity when $n \to \infty$, contradicting our assumption that $X \in
\Delta$.

(P2) Uniqueness of the fixed point: Let us suppose, by contradiction,
that there are two fixed points $X, Y \in \Delta$. They lie on the
opposite sides of $\Gamma$ (since $F_1$ and $F_2$ are hyperbolic). Let
$Y$ belong to the left half-plane. Let $W^u_F(X)$ denote the unstable
manifold of $F$ at $X$. By the proof of (P3), the negative eigenvalues
of both $DF(Y) = DF_1(Y)$ and $DF(X) = DF_2(X)$ are smaller than $-1$.
Consequently, for every $P, Q \in W^u_F(X)$, if the first coordinate
of $P$ is smaller than the first coordinate of $Q$, then this
inequality is reversed for $F(P)$ and $F(Q)$. Therefore, the distance
between $F^n(P)$ and $X$ goes to infinity when $n \to \infty$ for
every point $P \in W^u_F(X)$, $P \ne X$, contradicting our assumption
that $X \in \Delta$.

(P4) follows from the proof of (P2)-uniqueness.

(P5) follows from (P4) and (P3) similarly as in the proof of
(P2)-uniqueness.
\end{proof}

\section{Attractor}\label{sec:attractor}

Let $F$ be a Lozi-like map with the divider $\Gamma$ and a trapping
region $\Delta$. We will also use other notation introduced earlier.
We want to prove that $F$ restricted to the set $\Lambda :=
\bigcap_{n=0}^\infty F^n(\Delta)$ is topologically transitive, which
implies that $\Lambda$ is the attractor for $F$. Observe that
$\Lambda$ is completely invariant, that is,
$\Lambda=F(\Lambda)=F^{-1}(\Lambda)$.

From now on we will denote the unstable and stable manifold of a map
$f$ at a point $P$ by $W^u_f(P)$ and $W^s_f(P)$ respectively. Also, if
$A$ is an arc, or an arc-component and $P, Q \in A$, $P \ne Q$, we
denote by $[P, Q] \subset A$ a unique arc of $A$ with boundary points
$P$ and $Q$. Those sets will be usually subsets of $W^u_F(X)$,
$W^s_F(X)$, $\Gamma$ or $F(\Gamma)$, and we will call them sometimes
``segments.'' We will call the four regions of the plane given by
$\Gamma$ and $F(\Gamma)$ the \emph{quadrants}, and denote them by
$\Q_i$, $i=1,2,3,4$ (their order is the usual one). We will say that a
point $Q$ is above (below) a point $P$, if the vector from $P$ to $Q$
belongs to the upper (lower) half of the cone $K^s_P$. Analogously,
$Q$ is to the right (left) of $P$ if the vector from $P$ to $Q$
belongs to the right (left) half of the cone $K^u_P$.

Note first that $\Gamma \cap W^s_F(X) \ne \emptyset$. In the opposite
case $W^s_F(X) = W^s_{F_2}(X)$ and hence $W^s_{F_2}(X)$ would also not
intersect $F(\Gamma)$. But, (S3), (S1) and Lemma~\ref{cone3} imply
that $W^s_{F_2}(X)$ and $F(\Gamma)$ intersect in exactly one point, a
contradiction. Let us denote by $M$ the point where $W^s_{F_2}(X)$
intersects $\Gamma$. Note that $M$ is also a point of intersection of
$W^s_{F}(X)$ and $\Gamma$.

Recall that $\Gamma$ and $W^u_{F_2}(X)$ intersect at exactly one
point, denote it by $D$. Let us consider the arc $[D, F(D)] \subset
W^u_{F_2}(X)$, see Figure~\ref{fig.BP1}. Note that $F(D) = F_2(D) \in
F(\Gamma)$ and $X \in [D, F(D)] \subset W^u_F(X)$. Hence $F(D)$
belongs to the right half-plane and since the stretching factor
$\lambda$ is larger than $1$ (see Definition~\ref{df:sh}), $F^2(D)$
belongs to the second quadrant. Therefore $F^3(D)$ lies in the lower
half-plane. In order to prove that $\Lambda$ is the attractor for $F$,
we should restrict the possible position of $F^3(D)$, and increase the
lower bound on the stretching factor as follows:
\begin{enumerate}
\item[$(L3)$] $[F^3(D), F(D)] \subset W^u_F(X)$ intersects $[X, M]
  \subset W^s_F(X)$,
\item[$(L4)$] The stretching factor $\lambda$ is larger than $\sqrt2$.
\end{enumerate}
The above conditions are natural in the sense that any Lozi map with
$a$, $b$ as in~\cite{M} satisfies them.

\begin{figure}[h]
\begin{center}
\begin{tikzpicture}[scale=0.9]
\tikzstyle{every node}=[draw, circle, fill=white, minimum size=1.5pt, 
inner sep=0pt]
\tikzstyle{dot}=[circle, fill=white, minimum size=0pt, inner sep=0pt, outer sep=-1pt]

\node[dot, draw=none, label=above: \tiny $D$] at (0,1.9) {};
\node[dot, draw=none, label=above: \tiny $M$] at (0,-3.5) {};
\node[dot, draw=none, label=above: \tiny $N$] at (1,-2.5) {};
\node[dot, draw=none, label=above: \tiny $F(M)$] at (1,0.6) {};

\node[label=above: \tiny $F^4(D)$] at (-3.73,-1.64) {};
\node[] at (2.33,0.01) {};
\node[label=above: \tiny $F^2(D)$] at (-8.04,3.21) {};
\node[label=above: \tiny $F(D)$] at (6.42,-0.01) {};
\node[label=below: \tiny $F^3(D)$] at (-3.29,-4.02) {};
\node at (5.72,0) {};
\node[] at (0.25,1.16) {};
\node at (4.36,-0.01) {};
\node[label=below: \tiny $F^5(D)$] at (-1.98,-2.1) {};

\node[] at (0.37,-3.46) {};
\node[] at (0.75,-2.35) {};
\node[label=above: \tiny $X$] at (1.9,1) {};
\node[] at (-0.28,1.5) {};
\node[] at (1.55,-0.01) {};

\node[draw=none, label=above: \tiny $\Gamma$] at (-0.2,3) {};
\node[draw=none, label=above: \tiny $F(\Gamma)$] at (-8,0) {};

\draw[
decoration={snake,
	amplitude = 0.1mm,
	segment length = 7mm,
	post length=0.9mm},decorate] (-3.73,-1.64) -- (2.33,0);
\draw[
decoration={snake,
	amplitude = 0.1mm,
	segment length = 7mm,
	post length=0.9mm},decorate] (2.33,0) -- (-8.04,3.21);
\draw[
decoration={snake,
	amplitude = 0.1mm,
	segment length = 7mm,
	post length=0.9mm},decorate] (-8.04,3.21) -- (6.42,0);
\draw[
decoration={snake,
	amplitude = 0.05mm,
	segment length = 7mm,
	post length=0.9mm},decorate] (6.42,0) -- (-3.29,-4.02);
\draw[
decoration={snake,
	amplitude = 0.05mm,
	segment length = 7mm,
	post length=0.9mm},decorate] (-3.29,-4.02) -- (5.72,0);
\draw[
decoration={snake,
	amplitude = 0.03mm,
	segment length = 6mm,
	post length=0.9mm},decorate] (5.72,0) -- (0.25,1.16);
\draw[
decoration={snake,
	amplitude = 0.1mm,
	segment length = 7mm,
	post length=0.9mm},decorate] (0.25,1.16) -- (4.36,-0.01);
\draw[
decoration={snake,
	amplitude = 0.1mm,
	segment length = 7mm,
	post length=0.9mm},decorate] (4.36,-0.01) -- (-1.98,-2.1);
\draw[
decoration={snake,
	amplitude = 0.2mm,
	segment length = 15mm,
	post length=0.9mm},decorate] (-8.5,0) -- (7,0);
\draw[
decoration={snake,
	amplitude = 0.2mm,
	segment length = 10mm,
	post length=0.9mm},decorate](-0.5,3.3) -- (0.5,-4.3);
\draw[
decoration={snake,
	amplitude = 0.2mm,
	segment length = 10mm,
	post length=0.9mm},decorate](1.9,1) -- (0.37,-3.46);
\end{tikzpicture}
\caption{Positions of some distinguished points.}
\label{fig.BP1}
\end{center}
\end{figure}
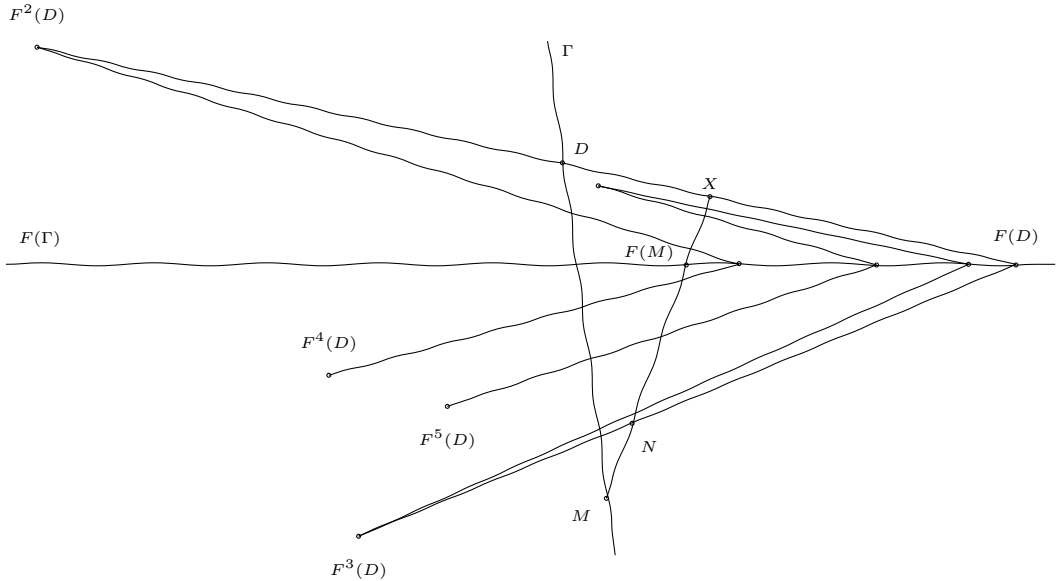

Let us denote by $N$ the point of intersection of $[F^3(D), F(D)]$ and
$[X, M]$. Let $\tilde H$ denote the ``quadrangle'' with vertices $D$,
$F(D)$, $N$, $M$, and edges $[D, F(D)]$, $[F(D), N] \subset W^u_F(X)$,
$[N, M] \subset W^s_F(X)$, and $[D, M] \subset \Gamma$. The set
$\tilde H$ is the union of two ``triangles.'' One of them has vertices
$X$, $N$, $F(D)$, and edges $[X, F(D)], [F(D), N] \subset W^u_F(X)$
and $[X, N] \subset W^s_F(X)$, and we denote it by $H_1$. The other
one has vertices $X$, $D$, $M$ and corresponding edges, see
Figure~\ref{fig.BP1}.

We will say that a smooth curve goes in the direction of the cone
$K^s$ or $K^u$ if vectors tangent to that curve are contained in the
corresponding cone.

\begin{lem}\label{lem:124}
If a Lozi-like map $F$ satisfies $(L3)$, then the intersection of
$\Lambda$ with the right half-plane is contained in $\tilde H$, and
the intersection of $\Lambda$ with the upper half-plane is contained
in $F(\tilde H)$.
\end{lem}

\begin{proof}
Let us show first that
\begin{equation}\label{e:124a}
F^{-1}(\Q_1\setminus\tilde H)\subset \Q_1\setminus\tilde H.
\end{equation}
To see that, observe where $F^{-1}$ maps the three pieces of the
boundary of $\Q_1\setminus\tilde H$. The part of $F(\Gamma)$ to the
right of $F(D)$ is mapped to the part of $\Gamma$ above $D$. The edge
$[D, F(D)]$ is mapped to its subset $[D,C]$, where $C=F^{-1}(D)\in
[X,F(D)]$. The part of $\Gamma$ above $D$ is mapped to a curve going
up from $C$ in the direction of the cone $K^s$ (and this curve cannot
intersect $\Gamma$). The set bounded by those images is contained in
$\Q_1\setminus\tilde H$, so~\eqref{e:124a} holds.

Now we show that
\begin{equation}\label{e:124b}
F^{-2}(\Q_4\setminus\tilde H)\subset \Q_1\setminus\tilde H.
\end{equation}
To see that, observe where $F^2$ maps the three pieces of the boundary
of $\Q_1\setminus\tilde H$. The part of $F(\Gamma)$ to the right of
$F(D)$ is mapped to a curve going to the left from $F^3(D)$ in the
direction of the cone $K^u$. The edge $[D, F(D)]$ is mapped to
$[F^2(D),F(D)]\cup[F(D),F^3(D)]$. The part of $\Gamma$ above $D$ is
mapped to a curve going to the left from $F^2(D)$ in the direction of
the cone $K^u$. The set $F^2(\Q_1\setminus\tilde H)$ lies in the right
part of the plane divided by those three curves. Because of the
condition (L3) and the form of $K^u$, this set contains the right
half-plane minus $\tilde H$. This proves~\eqref{e:124b}.

Suppose that a point $P$ is contained in the right half-plane and in
$\Lambda$, but not in $\tilde H$. Look at its trajectory for $F^{-1}$.
By~\eqref{e:124a} and~\eqref{e:124b}, $F^{-n}(P)$ is in the right
half-plane for all $n\ge 2$. Thus, the trajectory of $F^{-2}(P)$ for
$F^{-1}$ is the same as for $F_2^{-1}$. Since $\Lambda$ is bounded and
completely invariant, this trajectory must be bounded, so $F^{-2}(P)$
belongs to the unstable manifold of $X$ for the map $F_2$. Taking into
account that both $F^{-2}(P)$ and $F^{-3}(P)$ belong to the right
half-plane, we get $F^{-2}(P)\in[D,F(D)]$. However, $[D,F(D)]\subset
\tilde H$, and we get a contradiction. This proves that the
intersection of $\Lambda$ with the right half-plane is contained in
$\tilde H$.

Applying $F$ to both sides of this inclusion, and taking into account
that $\Lambda$ is completely invariant, we see that the intersection
of $\Lambda$ with the upper half-plane is contained in $F(\tilde H)$.
\end{proof}

Let $H := \bigcup_{n = 0}^{\infty}F^n(H_1)$.

\begin{lem}\label{lem:H}
Let a Lozi-like map $F$ satisfy $(L3)$. Then $\Lambda = \bigcap_{n =
  0}^{\infty}F^n(H)$.
\end{lem}

\begin{proof}
The definition of $H$ implies that $\partial H \subset [X, N] \cup
W^u_F(X)$ and $F(H) \subset H \subset \Delta$. Therefore, $\bigcap_{n
  = 0}^{\infty}F^n(H) \subseteq \bigcap_{n = 0}^{\infty} F^n(\Delta) =
\Lambda$.

Let us prove the reverse inclusion. Since $\Lambda$ is completely
invariant, it is sufficient to prove that $\Lambda \subset H$, and by
Lemma~\ref{lem:124} we only need to check the union of the third
quadrant and the set $(\tilde H \setminus H_1) \cap \Q_4$ (observe
that $F(\tilde H\subset H_1\cup F(H_1)\subset H$).

Suppose, by contradiction, that a point $P$ is contained in $\Lambda$
and in the lower half-plane, but not in $H$. Consider its trajectory
for $F^{-1}$. By the definition of $H$, $F^{-n}(P) \nin H$ for all $n
\ge 1$. Recall that $F^{-1}$ maps the lower half-plane onto the left
one. Assume that $F^{-n}(P)$ belongs to the left half-plane for all $n
\ge 1$. Since $H \cap \Q_2 = F(\tilde H) \cap \Q_2$,
Lemma~\ref{lem:124} implies that $F^{-n}(P)$ is in the third quadrant
for all $n \ge 1$. Then $F^{-n}(P) = F_1^{-n}(P)$ for all $n \ge 1$.
Since $\Lambda$ is bounded and completely invariant, this trajectory
is bounded, and hence belongs to the unstable manifold of the fixed
point $Y$ for the map $F_1$. Since $\Lambda$ is also closed, $Y \in
\Lambda$, a contradiction. Therefore, there exists $k \in \N$ such
that $F^{-k}(P)$ belongs to the right half-plane. Since $\Lambda$ is
completely invariant, Lemma~\ref{lem:124} implies that $F^{-k}(P) \in
\tilde H$ and hence $F^{-k+1}(P) \in F(\tilde H) \subset H$, a
contradiction.

This proves that $\Lambda$ is contained in $H$ and completes the
proof.
\end{proof}

\begin{rem}\label{rem:polygon}
\begin{enumerate}
\item[(1)] Note that the above proof implies that $H$ is a
  ``polygon.'' Namely, $H$ is the union of $H_1$ and ``triangles''
  $H_n$ defined inductively as $H_n := \Cl(F(H_{n-1}) \setminus H_1)$
  for $n \ge 2$. From the above proof it follows that there exists $m
  \ge 4$ such that $H_n = \emptyset$ for all $n \ge m$.
\item[(2)] The above proof also implies that $M$ lies below all points
  of $\Lambda \cap \Gamma$. Therefore, all points of $\Lambda \cap
  F(\Gamma)$ lie to the right of $F(M)$.
\end{enumerate}
\end{rem}

\begin{prop}\label{prop:closure}
Let a Lozi-like map $F$ satisfy $(L3)$. Then $\Lambda = \Cl
(W^u_F(X))$.
\end{prop}

\begin{proof}
Since $X \in \Delta$, $F(\Delta) \subset \Delta$, $\Delta$ is open and
$\Lambda = \bigcap_{n = 0}^{\infty}F^n(\Delta)$ is closed, we have
$\Cl (W^u_F(X)) \subseteq \Lambda$.

Let us prove the reverse inclusion. Let us suppose, by contradiction,
that there is a point $P\in\Lambda\setminus \Cl (W^u_F(X))$. Then
there exists $\eps > 0$ such that a ball with center $P$ and radius
$2\eps$ is disjoint from $W^u_F(X)$. Since $\partial H \subset [X, N]
\cup W^u_F(X)$ and $[X, N] \subset W^s_F(X)$, a ball with center $P$
and radius $\eps$ is disjoint from $\partial F^n(H)$, for all $n$
sufficiently large. Since the absolute value of the Jacobian of $F$ is
less than 1, the Lebesgue measures of $F^n(H)$ converge to $0$ as $n
\to \infty$. Therefore, $P \nin F^n(H)$ for $n$ sufficiently large.
Consequently, $P \nin \bigcap_{n = 0}^{\infty}F^n(H)$, contradicting
Lemma~\ref{lem:H}. This proves that $\Cl (W^u_F(X)) = \Lambda$.
\end{proof}

\begin{lem}\label{lem:intersect}
Let a Lozi-like map $F$ satisfy $(L3)$ and $(L4)$. Let $A \subset
W^u_F(X)$ be an arc. Then there exists $n \ge 0$ and a smooth arc
$\hat A \subset F^n(A)$ such that $\hat A$ intersects both $\Gamma$
and $F(\Gamma)$.
\end{lem}

\begin{proof}
Assume that such $n$ and $\hat A$ do not exist. The set $F^k(A)$ is an
arc which need not be smooth (that is, of class $C^1$), but is a union
of smooth arcs. We say that an arc $I \subseteq F^k(A)$ is
\emph{maximal smooth} if $I$ is smooth and no subarc of $F^k(A)$
containing $I$ (except $I$ itself) is smooth. If a maximal smooth arc
$I \subseteq F^k(A)$ intersects $\Gamma$, then $F(I)$ is a union of at
most two maximal smooth arcs, each of them intersecting $F(\Gamma)$.
By assumption $F(I)$ does not intersect $\Gamma$, and consequently
$F^2(I)$ consists of at most two maximal smooth arcs. If $I$ does not
intersect $\Gamma$, then $F(I)$ is a smooth arc and $F^2(I)$ consists
of at most two maximal smooth arcs. Hence, in both cases, $F^2(I)$
consists of at most two maximal smooth arcs. Thus, $F^{2k}(A)$
consists of at most $2^k$ maximal smooth arcs.

Since $\Lambda$ is bounded, there exists $z \in \R$ such that for
every smooth arc $I \subset W^u_F(X)$, the length of $I$ is smaller
than $z$. Therefore, the length of $F^{2k}(A)$ is not larger than $2^k
z$. On the other hand, $A \subset W^u_F(X)$, and hence the length of
$F^{2k}(A)$ is at least $\lambda^{2k}$ times the length of $A$. Thus,
for $k$ large enough we get $\lambda^2 \le 2$, which contradicts~(L4).
\end{proof}

\begin{rem}\label{rem:XM}
Since $M$ lies below $N$, each smooth arc contained in $H$ and
intersecting $\Gamma$ and $F(\Gamma)$ intersects $[X, M] \subset
W^s_F(X)$.
\end{rem}

\begin{prop}\label{prop:transitivity}
Let a Lozi-like map $F$ satisfy $(L3)$ and $(L4)$. Then $F|_{\Lambda}$
is topologically mixing, i.e., for all open subsets $U$, $V$ of $\R^2$
such that $U \cap \Lambda \ne \emptyset$ and $V \cap \Lambda \ne
\emptyset$, there exists $N \in \N$ such that for every $n \ge N$ the
set $F^n(U) \cap V \cap \Lambda$ is nonempty.
\end{prop}

\begin{proof}
Let $U$ and $V$ be open subsets of $\R^2$ such that $U \cap \Lambda
\ne \emptyset$ and $V \cap \Lambda \ne \emptyset$. By
Proposition~\ref{prop:closure}, $U \cap W^u_F(X) \ne \emptyset$ and $V
\cap W^u_F(X) \ne \emptyset$. Take a point $P \in V \cap W^u_F(X)$.
Since the points $F^{-k}(P)$ belong to $W^u_F(X)$ for every $k \ge 0$,
and the sequence $(F^{-k}(P))_{k=0}^\infty$ converges to $X$ as $k \to
\infty$, there exists $k_0$ such that the arc $[F^{-k_0}(P), X]
\subset W^u_F(X)$ is contained in the first quadrant. The map $F^{-1}$
is hyperbolic in the first quadrant and hence there exists $k_1 \ge
k_0$ and a neighborhood $V_1 \subset V$ of $P$, such that each arc of
$\Lambda$ contained in the first quadrant which intersects both
$\Gamma$ and $F(\Gamma)$, also intersects $F^{-k_1}(V_1)$. Moreover,
every $F^{-k}(V_1)$ for $k \ge k_1$ has the same property.

Since $U \cap W^u_F(X) \ne \emptyset$, there exists an arc $A \subset
U \cap W^u_F(X)$. By Lemma~\ref{lem:intersect}, there exists $n_1$
such that some subarc of $F^{n_1}(A)$ intersect both $\Gamma$ and
$F(\Gamma)$. Therefore, it also intersects $F^{-k}(V_1)$ for all $k
\ge k_1$. Hence, $F^n(U) \cap V \cap W^u_F(X) \supset
F^{n-n_1}(F^{n_1}(A) \cap F^{n_1-n}(V_1)) \ne \emptyset$ for all $n
\ge n_1 + k_1$.
\end{proof}

\section{Symbolic dynamics}\label{sec:symbolic}

In this section we assume that $F$ is a Lozi-like map defined by a
pair of synchronously hyperbolic $C^1$ diffeomorphisms $F_1$, $F_2$
and a divider $\Gamma$, with the attractor $\Lambda$.

We code the points of $\Lambda$ in the following standard way. To a
point $P \in \Lambda$ we assign a bi-infinite sequence $\bar p = \dots
p_{-2} \, p_{-1} \cd \, p_0 \, p_1 \, p_2 \dots$ of signs $+$ and $-$,
such that
\[
p_n =
\begin{cases}
-, & \textrm{if $F^n(P)$ is in the left half-plane},\\
+, & \textrm{if $F^n(P)$ is in the right half-plane}.
\end{cases}
\]
The dot shows where the 0th coordinate is.

A bi-infinite symbol sequence $\bar q = \dots q_{-2} \, q_{-1} \cd \,
q_0 \, q_1 \, q_2 \dots$ is called \emph{admissible} if there is a
point $Q \in \Lambda$ such that $\bar q$ is assigned to $Q$. We call
this sequence an \emph{itinerary} of $Q$. Obviously, some points of
$\Lambda$ have more than one itinerary. We denote the set of all
admissible sequences by $\Sigma_\Lambda$. It is a metrizable
topological space with the usual product topology. Since the left and
right half-planes (with the boundary $\Gamma$) that we use for coding,
intersected with $\Lambda$, are compact, the space $\Sigma_\Lambda$ is
compact.

In~\cite{MS} we did not prove the following two lemmas for the Lozi
maps, because they followed immediately from the results of
Ishii~\cite{I}. However, for the Lozi-like maps we cannot use those
results, so we have to provide new proofs.

\begin{lem}\label{lem:unique}
For every $\bar p \in \Sigma_\Lambda$ there exists only one point $P
\in \Lambda$ with this itinerary.
\end{lem}

\begin{proof}
Let $P,Q\in\Lambda$ be two points with the same itinerary $\bar p \in
\Sigma_\Lambda$. Let us define a non-autonomous dynamical system $G$
in $\R^2$ by the sequence $\bar p$ in the following way:
\[
G = (G_n)_{n \in \Z}, \quad G_n =
\begin{cases}
F_1, & \textrm{ if } p_n = -,\\
F_2, & \textrm{ if } p_n = +.
\end{cases}
\]
Note that $G^n(P) = F^n(P)$ and $G^n(Q) = F^n(Q)$ for every $n \in
\Z$. By~\cite[Proposition 5.6.1 (Hadamard-Perron)]{BS}, for the system
$G$ there exist stable and unstable manifolds of the points $P$ and
$Q$. By Lemma~\ref{cone3}, $W^s(P)$ and $W^u(Q)$ intersect at exactly
one point, $S$. If $S\ne P$, then
\[
\lim_{n\to-\infty}\|G^n(S)-G^n(P)\|=\infty\ \ \textrm{and}\ \
\lim_{n\to-\infty}\|G^n(S)-G^n(Q)\|=0,
\]
so
\[
\lim_{n\to-\infty}\|G^n(P)-G^n(Q)\|=\infty,
\]
a contradiction, because $\Lambda$ is completely invariant and
compact. Similarly, taking the limits as $n\to\infty$, we get a
contradiction if $S\ne Q$. Therefore, $P=S=Q$.
\end{proof}

By Lemma~\ref{lem:unique} and the definition of $\Sigma_\Lambda$, the
map $\pi : \Sigma_\Lambda \to \Lambda$ such that $\bar p$ is an
itinerary of $\pi(\bar p)$ is well defined and is a surjection.

\begin{lem}\label{lem:continuous}
The map $\pi$ is continuous.
\end{lem}

\begin{proof}
Let $\bar p \in \Sigma_\Lambda$ and let $(\bar p^n)_{n=1}^\infty$ be a
sequence of elements of $\Sigma_\Lambda$ which converges to $\bar p$.
Set $P=\pi(\bar p)$ and $P^n=\pi(\bar p^n)$. We will prove that
$(P^n)_{n=1}^\infty$ converges to $P$.

Let us suppose by contradiction, that $(P^n)_{n=1}^\infty$ does not
converge to $P$. Since $\Lambda$ is compact, there exists a
subsequence of the sequence $(P^n)_{n=1}^\infty$ convergent to some $Q
\ne P$. We may assume that this subsequence is the original sequence
$(P^n)_{n=1}^\infty$.

If for every $j\in\Z$ both $F^j(P)$ and $F^j(Q)$ belong to the same
closed half-plane, then there is $\bar q\in\Sigma_\Lambda$ which is an
itinerary of both $P$ and $Q$. This contradicts
Lemma~\ref{lem:unique}. Therefore, there is $j\in\Z$ such that
$F^j(P)$ and $F^j(Q)$ belong to the opposite open half-planes. Since
the points $P^n$ converge to $Q$, for all $n$ large enough the points
$F^j(P^n)$ and $F^j(P)$ belong to the opposite open half-planes. This
means that for all $n$ large enough we have $p^n_j\ne p_j$. Therefore
the sequence $(\bar p^n)_{n \in \N}$ cannot converge to $\bar p$, a
contradiction. This completes the proof.
\end{proof}

As we mentioned in the introduction, the rest of the results
of~\cite{MS} holds for Lozi-like maps satisfying conditions (L3) and
(L4), and the proofs are practically the same. Therefore we will not
repeat those results here, and will send the reader to~\cite{MS}.

However, in the next section we will be speaking about the
\emph{kneading sequences}, so we have to define them. They are
itineraries of the \emph{turning points}, that is, points of the
intersection $W^u_F(X)\cap F(\Gamma)$. In fact, we will use only the
nonnegative parts of those itineraries, that is the usual sequences
$(p_0,p_1,p_2,\dots)$. In this case, we do not have to worry about a
nonuniqueness of an itinerary. If $P$ is a turning point and
$F^n(P)\in\Gamma$ for some $n>0$, then some neighborhood of $F^n(P)$
along $W^u_F(X)$ belongs to the same half-plane (right or left). In
such a case we will accept only the corresponding symbol ($+$ for
right or $-$ for left) as $p_n$.

\section{Example}\label{sec:ex}

In this section we will give an example of a three-parameter family of
Lozi-like maps, containing the two-parameter family of Lozi maps. One
can expect that it is essentially larger than the Lozi family. We will
provide a strong numerical evidence that some maps in this family have
the set of kneading sequences that does not appear in the Lozi family.

Let us consider the following family of maps:
$F_{a,b,c} : \R^2 \to \R^2$,
\[
F_{a,b,c}(x, y) =  (1 + y - a|x| - cx, bx).
\]
We will use three assumptions on the parameters $a$, $b$ and $c$.
\begin{enumerate}
\item[(A1)] $0 < b < 1$, $c \ge 0$
\item[(A2)] $(2a + b)\left(1 - \frac{c^2}{(a + b)^2}\right) < 4$
\item[(A3)] $(a - c)\sqrt{2} > b+2$
\end{enumerate}

\begin{figure}
	\centering
	\includegraphics[height=8cm]{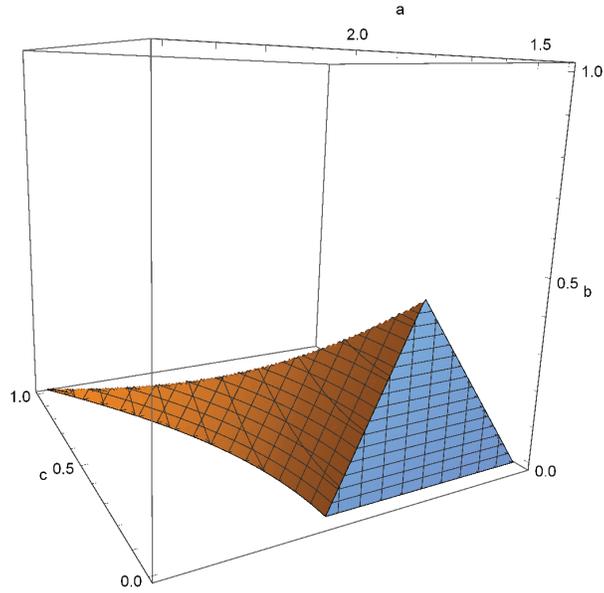}
	\caption{The set of parameters.}
	\label{PS}
\end{figure}

Obviously, for $c = 0$ we get the Lozi family, $F_{a,b,0} = L_{a,b}$.
Let $c \ne 0$, and fix $a$, $b$ and $c$ which satisfy the above
assumptions. For simplicity, let $F := F_{a,b,c}$. We will show that
$F$ is a Lozi-like map.

\begin{prop}\label{prop:ps}
Assume that the parameters $a,b,c$ satisfy {\rm(A1)--(A3)}. Then the
following hold:
\begin{enumerate}[{\rm(C1)}]
\item $2a + b < 2(1 + \sqrt{1 + c^2})$;
\item $c < \min \{ a-b-1, 1 \}$ and $a > b+1$;
\item If $a \ge 1$ then
\[
\frac{(a-c)(2a^2+3a+(2a+1)c)}{4(a+1)^2} \ge
\frac{7}{16}a-\frac{3c^2+2c+2}{16};
\]
\item $\displaystyle{b \le \min \left\{
  \frac{(a-c)(2a^2+3a+(2a+1)c)}{4(a+1)^2}, \frac12 \right\}}$;
\item $a^3 - 4a + a^2c - ac^2 - c^3 - 4bc \ge (-a^2 + 2b +
c^2)\sqrt{(a+c)^2 + 4b}$.
\end{enumerate}
\end{prop}

We prove this theorem in Appendix~\ref{AppZ}.

\bigskip\noindent{\bf Basic properties.} Let $F_1, F_2 : \R^2 \to
\R^2$ be as follows:
\[
F_1(x, y) =  (1 + y + (a-c)x, bx), \ F_2(x, y) =  (1 + y - (a+c)x,
bx).
\]
We will first show that $F_1$ and $F_2$ are synchronously hyperbolic.
Let $\Gamma$ be the $y$-axis. Then $F_1(\Gamma) = F_2(\Gamma)$ is the
$x$-axis. Moreover, for every point $P \in \Gamma$, $F_1(P) = F_2(P)$.
The maps $F_1$ and $F_2$ are linear and each of them maps the left
half-plane onto the lower one and the right half-plane onto the upper
one.

The fixed point of $F_1$ is $Y = \left(\frac{1}{1-a-b+c},
\frac{b}{1-a-b+c}\right)$, and by~(C2), it is in the third quadrant.
The fixed point of $F_2$ is $X = \left(\frac{1}{1+a-b+c},
\frac{b}{1+a-b+c}\right)$ and it is in the first quadrant. Both maps
are hyperbolic. The eigenvalues of $DF_1$ are $\hat \lambda_1 =
\frac12\left(a-c+\sqrt{(a-c)^2+4b}\right)$ and $\hat \lambda_2 =
\frac12\left(a-c-\sqrt{(a-c)^2+4b}\right)$, and by~(C2), $\hat
\lambda_1 > 1$ and $-1 < \hat \lambda_2 < 0$. The eigenvalues of
$DF_2$ are $\lambda_1 = \frac12\left(-a-c-\sqrt{(a+c)^2+4b}\right)$
and $\lambda_2 = \frac12\left(-a-c+\sqrt{(a+c)^2+4b}\right)$, and
$\lambda_1 < -1$ and $0 < \lambda_2 < 1$. The eigenvector
corresponding to an eigenvalue $\lambda$ is $\langle \lambda, b
\rangle$. Also, $\det DF_i(P) = -b$ for every point $P \in \R^2$ and
$i = 1, 2$, and by~(A1), $-1 < \det DF_i(P) < 0$, that is,~(L1) holds.

\bigskip\noindent{\bf Universal cones.}
The derivative of $F_i$, $i = 1,2$, is
\[
DF_i = \begin{bmatrix}
	\pm a - c & 1 \\ b & 0
\end{bmatrix},
\]
where the sign depends on $i$, for $i=1$ the sign is $+$ and for $i=2$
the sign is $-$. Let
\[
\begin{bmatrix}
\pm a - c & 1 \\ b & 0
\end{bmatrix}
\begin{bmatrix}
x \\ y
\end{bmatrix}
= \begin{bmatrix}
x' \\ y'
\end{bmatrix}.
\]
We want to find constants $0 < \lambda' < 1 < \lambda''$ such that
$DF_i^{-1}$, $i = 1,2$, maps the stable cone
\begin{equation*}\label{eq:scone}
K^s = \{ \langle x', y' \rangle : \lambda''|x'| \le |y'| \}
\end{equation*}
into itself and expands all its vectors by a factor larger than 1, and
$DF_i$, $i = 1,2$, maps the unstable cone
\begin{equation*}\label{eq:ucone}
K^u = \{ \langle x, y \rangle : |y| \le \lambda'|x| \}
\end{equation*}
into itself and expands all its vectors by a factor larger than 1.

Set $d = \frac12\left(a-c-\sqrt{(a-c)^2-4b}\right)$. By~(A1) and~(A3),
$(a-c)^2-4b > 0$. By~(C2), $d < b$ and consequently
\begin{equation}\label{eq:d}
0 < d < b < 1 < \frac{b}{d}.
\end{equation}

One can prove (see Appendix~\ref{AppW}) that
\begin{equation}\label{eq:stable}
b|x'| \le d|y'| \ \Rightarrow \ b|x| \le d|y|, \ |x'| \le
d|x|\ \ \textrm{and}\ \ |y'| \le d|y|,
\end{equation}
\begin{equation}\label{eq:unstable}
|y| \le d|x| \ \Rightarrow \ |y'| \le d|x'|, \ |y'| \ge
\frac{b}{d}|y|\ \ \textrm{and}\ \ |x'| \ge \frac{b}{d}|x|.
\end{equation}
Consequently, we can take $\lambda' = d$ and $\lambda'' = b/d$.
By~\eqref{eq:d} both $1/d$ and $b/d$ are larger than 1. Note that the
$y$-axis is the axis of $K^s$ and the $x$-axis is the axis of $K^u$.
Since $b > d^2$, the cones $K^s$ and $K^u$ are disjoint. Therefore, by
Definition~\ref{df:uc}, $K^s$ and $K^u$ are the universal pair of
cones. Also, by Definition~\ref{df:sh}, the maps $F_1$ and $F_2$ are
synchronously hyperbolic.

\bigskip{\bf Trapping region.} We will use here notation introduced in
Section~\ref{sec:attractor}. Let us recall that the point where
$W^u_{F_2}(X)$ intersects $\Gamma$ is denoted by $D$. Let $(x_i,
y_i):=F^i(D)$. The point where $W^s_{F_2}(X)$ intersects $\Gamma$ is
denoted by $M = (x_M, y_M)$, see Figure~\ref{fig.BP2}.

Calculating, we obtain
\[
F(D) = \left(\frac{2 + a + c + \sqrt{(a + c)^2 + 4b}}
{2(1 + a - b + c)}, 0 \right),
\]
\[
F^2(D) = \left(\frac{2 - 2 b - (a + c)\left(a + c + \sqrt{(a + c)^2 +
    4b}\right)}{2(1 + a - b + c)}, \frac{b\left(2 + a + c + \sqrt{(a +
    c)^2 + 4b}\right)}{2(1 + a - b + c)} \right).
\]
By~(A1), (A3) and~(C2) it follows that $x_1 > 0$, $x_2 < 0$ and $y_2 >
0$, that is, $F(D)$ belongs to the right half-plane and $F^2(D)$
belongs to the second quadrant. Denote by $B$ the point of
intersection of the $y$-axis with the union of segments $[F(D),
F^3(D)] \cup [F^3(D), F^2(D)]$ (remember that it may happen that
$F^3(D)$ lies in the right half-plane). Conditions~(C4) and~(C5)
imply~(L3), that is, $[F^3(D), F(D)] \cap [X, M] \ne \emptyset$.
Namely,~(C5) implies that $F^3(D)$ lies to the left of the line
through $X$ and $M$, and~(C4) implies that $B$ is above $M$, that is,
$B \in [D, M]$ (for more details see Appendices~\ref{AppA}
and~\ref{AppB}). Moreover,~(A3) implies~(L4), that is, the stretching
factor $\lambda=b/d$ is larger than $\sqrt2$.

Finally, let us prove~(L2), that is, that there exists a trapping
region $\Delta$ (for the map $F$) which is homeomorphic to an open
disk and its closure is homeomorphic to a closed disk.

Let us consider the triangle $\Theta$ with vertices $F^i(D)$, $i =
1,2,3$.
\begin{figure}[h]
\begin{center}
		\begin{tikzpicture}[scale=1.4]
		\tikzstyle{every node}=[draw, circle, fill=white, minimum size=1.5pt,
		inner sep=0pt]
		\node[draw=none, label=above: \tiny $\Gamma$] at (-0.2,2) {};
		\node[draw=none, label=above: \tiny $F(\Gamma)$] at (-4.8,0) {};
		\node[draw=none, label=above: \tiny $D$] at (-0.2,1) {};
		\node[draw=none, label=above: \tiny $B$] at (-0.2,-2) {};
		\node[draw=none, label=above: \tiny $M$] at (-0.2,-2.4) {};
		\node[draw=none, label=above: \tiny $F(B)$] at (1.5,-0.1) {};

		\node[label=above: \tiny $F^4(D)$] (n1) at
		(3*-0.746535, 3*-0.329325) {};
		\node[] (n2) at (3*0.46796, 0) {};
		\node[label=above: \tiny $F^2(D)$] (n3) at
		(3*-1.60933, 3*0.642692) {};
		\node[label=above: \tiny $F(D)$] (n4) at
		(3*1.28538, 0) {};
		\node[label=below: \tiny $F^3(D)$] (n5) at
		(3*-0.65865, 3*-0.804665) {};
		\node (n6) at (3*1.14478, 0) {};
		\node[] (n7) at (3*0.0500411, 3*0.23398) {};
		\node (n8) at (3*0.873107, 0) {};
		\node[label=below: \tiny $F^5(D)$] (n9) at
		(3*-0.39687, 3*-0.373268) {};

		\node[] at (0,3*-0.53204) {};
		\node[] (n14) at (0,3*-0.692499) {};
		\node[label=above: \tiny $X$] (n15) at
		(3*0.395257,3*0.197628) {};
		\node[label=above: \tiny $Y$] at
		(3*-1.07527,3*-0.537634) {};
		\node[] at (0,3*0.2853845) {};

		\node[draw=none] (n10) at (4.5,0) {};
		\node[draw=none] (n11) at (-5,0) {};

		\node[draw=none] (n12) at (0,-2.5) {};
		\node[draw=none] (n13) at (0,2.2) {};

	    \foreach \from/\to in {n1/n2, n2/n3, n3/n4, n4/n5, n5/n6,
	    	n6/n7, n7/n8, n8/n9, n10/n11, n12/n13, n14/n15,
	    n3/n5, n5/n1}
	    \draw (\from) -- (\to);
		\end{tikzpicture}
	\caption{The triangle $\Theta$ and positions of some
		distinguished points.}
	\label{fig.BP2}
\end{center}
\end{figure}
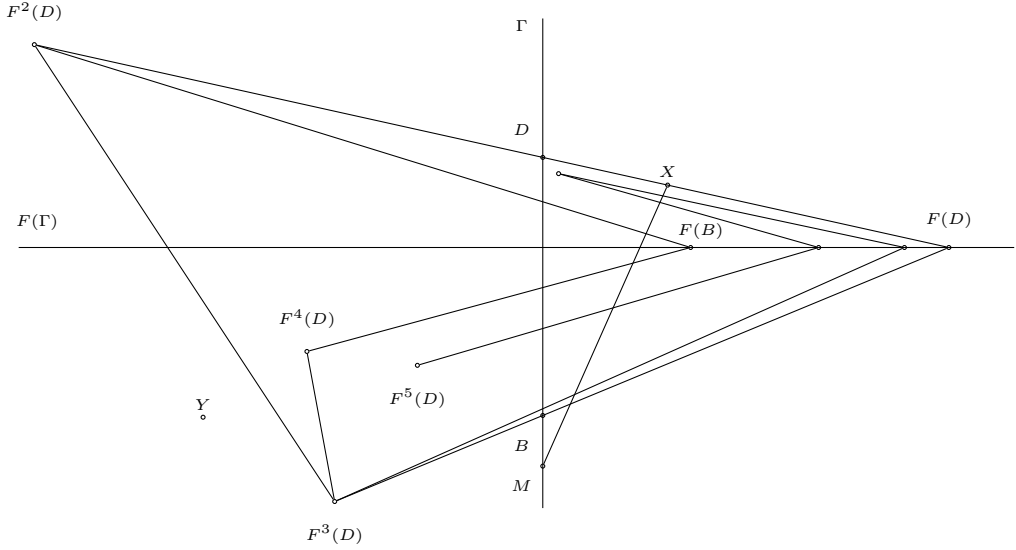

\begin{lem}\label{lem:tau}
Let $a$, $b$, $c$ satisfy {\rm (A1)--(A3)}. Then $F(\Theta) \subset
\Theta$.
\end{lem}

\begin{proof}
Since $F$ is linear on both the left and the right half-planes, the
set $F(\Theta)$ lies above the line through the points $F^3(D)$ and
$F(D)$, and below the line through the points $F^2(D)$ and $F(D)$. The
condition $B \in [D, M]$ implies $F(B) \in [F(M), F(D)]$. Recall that
by~(C5), $F^3(D)$ lies to the left of the line through $X$ and $M$.
Therefore $F(M)$, which is the point of intersection of $[X, M]$ and
$F(\Gamma)$, belongs to $\Theta$. Thus $[F(M), F(D)] \subset \Theta$.

The set $F(\Theta)$ is a pentagon with vertices $F^i(D)$, $i =
1,2,3,4$ and $F(B)$. If $F^4(D) \in \Theta$, then they all belong to
$\Theta$ and, consequently, $F(\Theta) \subset \Theta$. By~(C2)
and~(C4), if $F^3(D)$ lies in the right half-plane, then $F^3(D)$
belongs to the triangle with vertices $X$, $M$, $D$, and hence
$F^4(D)$ belongs to the triangle with vertices $X$, $F(M)$, $F(D)$
which is contained in $\Theta$. If $F^3(D)$ lies in the left
half-plane, then~(A2) implies that $F^4(D)$ lies to the right of the
line through the points $F^2(D)$ and $F^3(D)$ (see
Appendix~\ref{AppC}), and this completes the proof.
\end{proof}

Now we can define a trapping region in the same way as in~\cite{M}. We
have
\[
F([F^2(D), F^3(D)] \cup [F^3(D), F(D)] \setminus \{ F^2(D), F(D) \})
\subset \Int \Theta,
\]
\[
F([F^2(D), D]) = [F^3(D), F(D)].
\]
Therefore, $F^{-4}(\Theta)$ is a neighborhood of
\[
[F^3(D), F(D)] \cup [F^2(D), F^3(D)] \cup [F^2(D), D].
\]
Since $[D, F(D)]$ is a local unstable manifold of the hyperbolic fixed
point $X$, there exists a rectangle $R$ contained in the first
quadrant, with the sides parallel to the eigenvectors of $DF(X)$, such
that
\[
\Int R \supset [D, F(D)] \setminus F^{-4}(\Int \Theta)
\]
and
\[
F(R) \subset \Int R \cup F^{-4}(\Int \Theta).
\]
Define $\Delta := R \cup F^{-4}(\Theta)$. The set $\Delta$ is a
compact neighborhood of $\Theta$. Let us show that it is a trapping
region.

\begin{lem}\label{lem:tau1}
Let $a$, $b$, $c$ satisfy {\rm (A1)--(A3)}. Then $F(\Delta) \subset
\Int\Delta$.
\end{lem}

\begin{proof}
By the definition, $F(R) \subset \Int \Delta$. Also $F^4(R) \supset
[F(D), F^2(D)] \cup [F(D), F^3(D)]$. Therefore, $F(\Theta \setminus
F^4(R)) \subset \Int \Theta$. Thus,
\[
F(\Delta \setminus R) = F(F^{-4}(\Theta)\setminus R) = F^{-4}(F(\Theta
\setminus F^4(R))) \subset F^{-4}(\Int \Theta) \subset \Int \Delta.
\]
Consequently, $F(\Delta ) \subset F(R) \cup F(\Delta \setminus R)
\subset \Int \Delta$.
\end{proof}

Obviously, $\Delta$ is homeomorphic to a closed disk, its interior is
homeomorphic to an open disk, so we have proved that $F$
satisfies~(L2). Therefore, $F$ is Lozi-like.

\bigskip\noindent{\bf Differences between the families $L_{a,b}$ and
  $F_{a,b,c}$.} Recall that $L_{a,b} = F_{a,b,0}$. We want to show
that the family $F_{a,b,c}$ is essentially larger than the family
$L_{a,b}$. That is, we want to find parameters $a', b', c'$ such that
the set of kneading sequences of $F_{a',b',c'}$ is different from the
set of kneading sequences of any $L_{a,b}$.

We will not prove that rigorously, but we will present a very strong
numerical evidence. We will comment on the reliability of our
computations later.

By~(C2), we have $c^2<a^2+(1-b)^2$, so there is a unique
point $Q$ such that $Q$ is in the right half-plane, $F(Q)$ is in the
left half-plane, and $F^2(Q)=Q$. Let us consider the maps $F_{a,b,c}$
with the following properties:
\begin{enumerate}
\item[(F1)] The point $F^3_{a,b,c}(D)$ lies in the left half-plane and
  $F^4_{a,b,c}(D) \in [X, M] \subset W^s_F(X)$.
\item[(F2)] The point $F(B)$ lies on the stable manifold of $F_1\circ
  F_2$ of $Q$.
\end{enumerate}

Assumption~(F1) means that the nonnegative part of the largest
kneading sequence is $+--\rpinf$. Assumption~(F2) means that the
nonnegative part of the smallest kneading sequence is $(+-)^\infty$.
Note that compared to Figure~\ref{fig.BP2}, there is a difference:
$F^4(D)$ lies in the fourth quadrant, and $F^5(D)$ lies in the first
quadrant (see Figure~\ref{lc0}).

\begin{figure}
	\begin{center}
	\includegraphics[height=7cm]{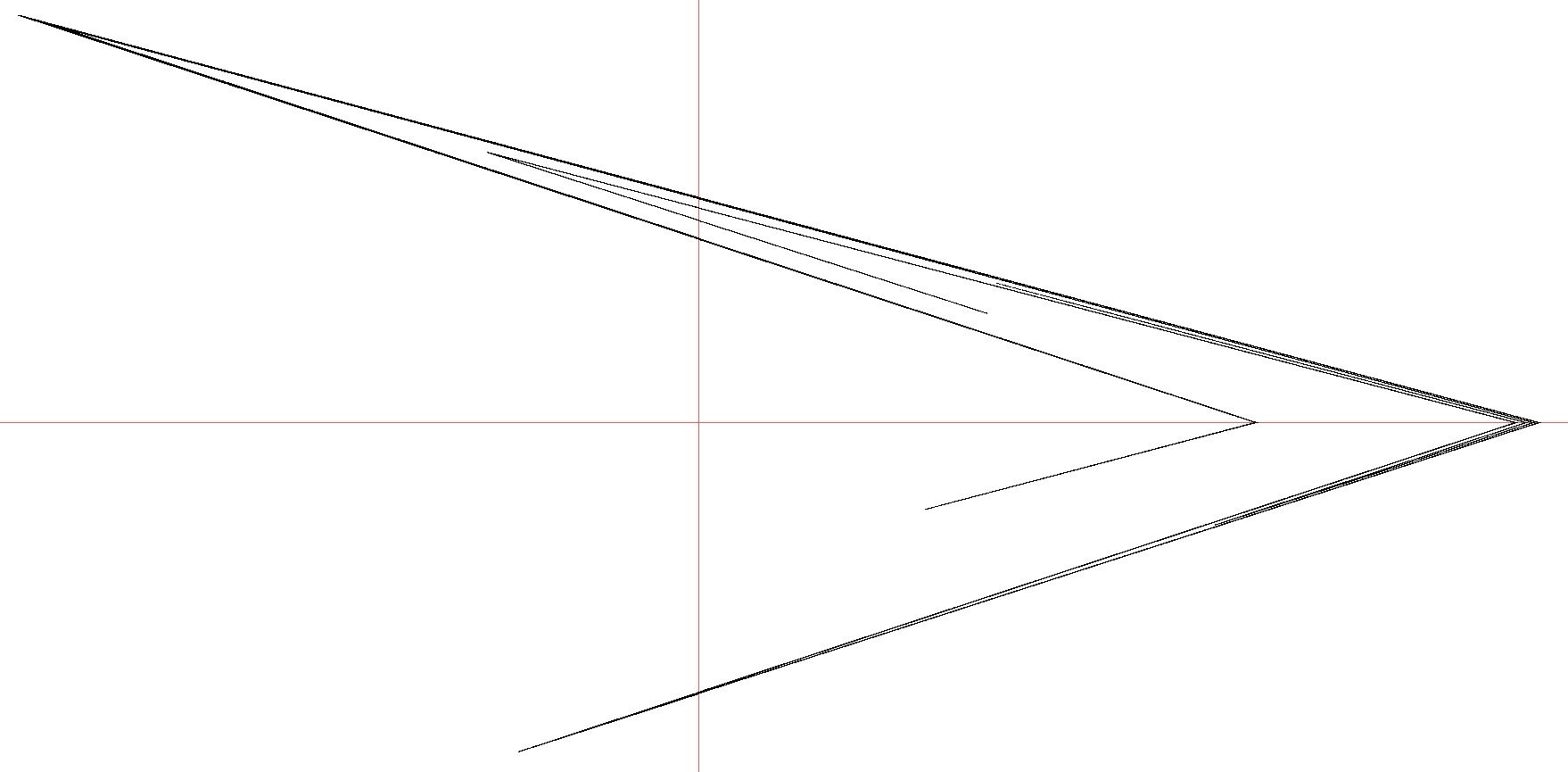}
	\caption{Attractor for the Lozi map with parameters described
          by~(F1') and~(F2'). The $y$-coordinate is stretched by
          factor $7/4$.}\label{lc0}
        \end{center}
\end{figure}

Assumption~(F1) gives the equation
\begin{equation}\label{F1}
\begin{split}
(a^2 - c^2)^2 - 6a^2 -4a & + 4 b^2 + a^2b + c(2a - 2ab + 5bc)\\ &
+ (a^3 + 2a - ab +c^3 - a^2c - ac^2 + 3bc)\sqrt{4 b + (a + c)^2} = 0
\end{split}
\end{equation}
Assumption~(F2) gives the equation
\begin{equation}\label{F2}
\begin{split}
-1 & + \frac{1 + a - b - c}{a^2 + (-1 + b)^2 - c^2} + \frac{(-1 + a +
  b + c)(a^2 - c^2 - \sqrt{(a^2 - c^2)(a^2 - 4 b - c^2)})} {2(a +
  c)(a^2 + (-1 + b)^2 - c^2)}\\ &\phantom{mmmmm}
 - \frac{2b}{-2 b + c + \sqrt{4 b + (a
    + c)^2} - a(3 + 2a + 2c - 2\sqrt{4 b + (a + c)^2})} = 0
\end{split}
\end{equation}
(For more details, see Appendix~\ref{AppD}.)
For $c=0$ these two equations are (we simplify the second equation)
\begin{enumerate}
\item[(F1')] $a^4 - 6a^2 -4a + 4 b^2 + a^2b + (a^3 + 2a - ab)\sqrt{4 b
  + a^2} = 0$,
\item[(F2')] $\displaystyle{4\cdot\frac{-a^2-2b^2+2b+a
    \sqrt{a^2-4b}}{a-2b-\sqrt{a^2-4b}}-
  \left(2+a-\sqrt{a^2+4b}\right)\left(3a-\sqrt{a^2+4b}\right)=0}.$
\end{enumerate}

Computer plots of the graphs of the above equations are presented in
Figure~\ref{f:plot}. The graphs are smooth and evidently intersect
each other at one point (although we cannot claim that we proved it).
Moreover, using the ``NSolve'' command of Wolfram Mathematica produces
a unique solution to this system of equations in the region
$a\in[1,2]$, $b\in[0,1]$. This solution is approximately
\[
a = 1.65531960296885174459210852526,\ \ \
b = 0.276507107967726099812119447619.
\]

The problem is that we do not know how the computer produces graphs or
solves a system of equations. Thus, it could happen that the graphs
have more branches. To eliminate this possibility, we checked for
both~(F1') and~(F2') whether the left-hand side is positive or
negative (in the region mentioned above), see Figure~\ref{f:plot1}.
This method would identify (although with limited precision)
additional solutions. However, it gave us the same result as before.

\begin{figure}
	\begin{center}
	\includegraphics[height=8cm]{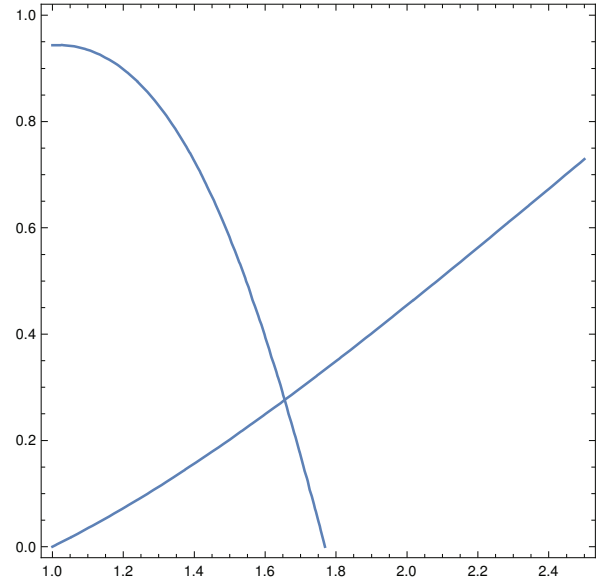}
	\caption{Graphs of (F1') and (F2').}\label{f:plot}
        \end{center}
\end{figure}

\begin{figure}
	\begin{center}
	\includegraphics[height=8cm]{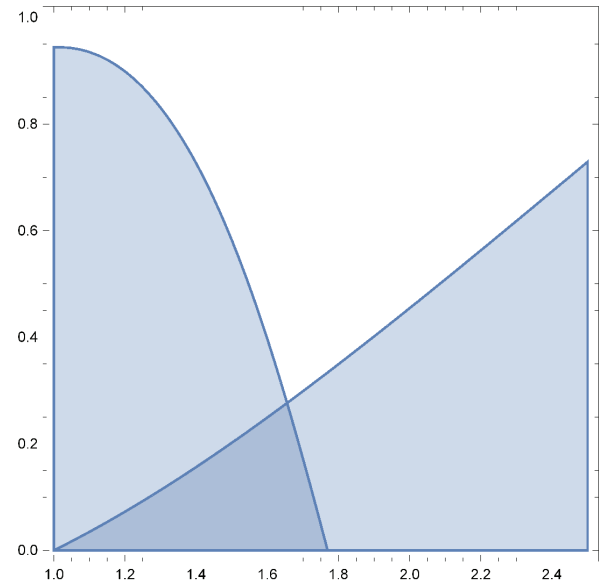}
	\caption{Equations (F1') and (F2') as inequalities.}\label{f:plot1}
        \end{center}
\end{figure}

For $c=0.1$ the values of $a$ and $b$ are approximately
\[
a = 1.63537454884191587958622457986,\ \ \
b = 0.276988367360779957370639853557.
\]
Here we do not have to worry about the uniqueness. We just need some
values of parameters. One can check that in both cases, $c=0$ and
$c=0.1$, the map satisfies conditions~(A1)--(A3).

We computed kneading sequences using Free Pascal. Re-checking
numerically nonnegative parts of the largest and smallest kneading
sequences in both cases, we get correct signs for at least 68
iterates. Since the maximal stretching factor of $F$ is about $1.8$,
and the precision of our floating point computations is about 19
decimal places, this is what we could expect.

Now we look at the nonnegative parts of the kneading sequence of the
turning point $S$ which is the next after $F(D)$ along the unstable
manifold of $X$ in that direction (of course, $S$ depends on $c$). In
Figure~\ref{lc0}, this turning point is the leftmost of the right
group of turning points. For $c=0$, the nonnegative part of the
kneading sequence starts with $+--+++-+-+-+++-++$, while for $c=0.1$
it starts with $+--+++-+-+-+++++-$. We see a difference at the place
corresponding to $F^{14}(S)$. The distance of this point from the
divider is about $0.6$ for $c=0$ and about $8\cdot 10^{-4}$ for
$c=0.1$ (if one wants much larger distances, they appear for
$F^{16}(S)$). The roundoff error, even if we take into account
accumulation of errors, should not be larger than $10^{-14}$, so the
results are quite reliable.

\begin{rem}
Note that even for the case $c=0$ (that is, for the Lozi maps)~(C4)
and~(C5) imply~(L3). In~\cite{M},~(L3) was obtained by a stronger
condition, $b<(a^2-1)/(2a+1)$. By replacing this condition by our part
of (C4), $b \le \frac{(a-c)(2a^2+3a+(2a+1)c)}{4(a+1)^2}$, we get a
slightly larger region in the parameter plane for the Lozi family,
where the hyperbolic attractor exists. This new region is the triangle
bounded by the lines $b=0$, $2a+b=4$, and $a\sqrt2=b+2$, according
to~(A1)--(A3). The gain is a small region close to the top of this
triangle, that was cut off in~\cite{M}.
\end{rem}

\begin{rem}
Although we were concerned only with the nonnegative parts of the
kneading sequences, it is clear that the whole largest (respectively,
smallest) kneading sequences are the same for the cases $c=0$ and
$c=0.1$.
\end{rem}

\appendix
\section*{Appendix}\label{sec:A}
\renewcommand{\thesubsection}{\Alph{subsection}}

\subsection{Proof of Theorem~\ref{prop:ps}}\label{AppZ}

\noindent\textbf{(C1)} Suppose that $2a + b \ge 2(1 + \sqrt{1 +
  c^2})$. Then, by~(A2), $2(1 + \sqrt{1 + c^2})(1 - \frac{c^2}{(a +
  b)^2}) < 4$, and hence, $\frac{c^2}{(a + b)^2} > 1 - \frac{2}{1 +
  \sqrt{1 + c^2}} = \frac{c^2}{(1 + \sqrt{1 + c^2})^2}$. Therefore, $a
+ b < 1 + \sqrt{1 + c^2}$. Then $2(1 + \sqrt{1 + c^2}) \le 2a + b < a
+ (1+\sqrt{1 + c^2})$, so $a > 1 + \sqrt{1 + c^2}$, and hence $1 +
\sqrt{1 + c^2} < a < a+b < 1 + \sqrt{1 + c^2}$, a contradiction.

\medskip\noindent\textbf{(C2)} Since $0 < b < 1$, we have
$-\frac{b+2}{\sqrt{2}} < -b-1$, so $c < a - \frac{b+2}{\sqrt{2}} <
a-b-1$. In particular, $a-b-1 > 0$. Now we will prove that $c < 1$.
By~(C1), $2a < 2 + 2\sqrt{1+c^2} - b$. By~(A3), $2a > 2c + \sqrt2 b +
2\sqrt2$. Thus, $2c + \sqrt2 b + 2\sqrt2 < 2 + 2\sqrt{1+c^2} - b$, so
$0 < (\sqrt2 + 1)b < 2(1 + \sqrt{1+c^2} - c - \sqrt2 )$. Set $\phi(c)
= 1 + \sqrt{1+c^2} - c - \sqrt2$. Then $\phi(1) = 0$ and $\phi'(c) =
\frac{c}{\sqrt{1+c^2}} - 1 < 0$. Thus, $c < 1$.

\medskip\noindent\textbf{(C3)} We
have $$\frac{(a-c)(2a^2+3a+(2a+1)c)}{4(a+1)^2} = \frac{a}{2} - \frac14
-\frac{2ac+(2a+1)c^2-1}{4(a+1)^2}.$$ Thus, we only need to prove that
if $a \ge 1$ then
\begin{equation}\label{e:ps3}
\frac{2ac+(2a+1)c^2-1}{4(a+1)^2} \le \frac{3c^2+2c+2}{16} - \frac14 +
\frac{a}{16}.
\end{equation}
If $a=1$,~\eqref{e:ps3} is an equality. The derivative with respect to
$a$ of the left-hand side is
\[
\frac{-ac^2-ac+c+1}{2(a+1)^3} \le \frac{1}{2(a+1)^3} \le \frac{1}{16}
\]
if $a \ge 1$, while the derivative of the right-hand side is
$\frac{1}{16}$. Thus, the inequality~\eqref{e:ps3} holds.

\medskip\noindent\textbf{(C4)} We know by~(A3) that
\begin{equation}\label{ps5a}
\sqrt2 a - b > \sqrt2 c + 2.
\end{equation}
Also, by~(C1), $2a+b < 2\left(1 + \sqrt{1+c^2}\right)$. However,
$\sqrt{1+c^2} \le 1+c$, so
\begin{equation}\label{ps5b}
2a + b < 4 + 2c.
\end{equation}
By~(C3), to prove that $b \le
\frac{(a-c)(2a^2+3a+(2a+1)c)}{4(a+1)^2}$, we only need to prove that
\begin{equation}\label{ps5c}
b < \frac{7}{16}a - \frac{3c^2+2c+2}{16}.
\end{equation}

Consider three lines in the $(a,b)$-plane given by equalities
in~\eqref{ps5a},~\eqref{ps5b} and~\eqref{ps5c}. Taking into account
the slopes of those lines, we see that if $(a_0, b_0)$ is the point of
intersection of two first lines, it is enough to prove
that~\eqref{ps5c} holds for $a = a_0$, $b = b_0$. We have $a_0 =
\frac{6}{2+\sqrt2} + c$ and $b_0 = 4 - \frac{12}{2+\sqrt2} < \frac12$,
so, in particular, $b < \frac12$. Hence, it is enough to prove that
$\frac12 < \frac{7}{16}\left(\frac{6}{2+\sqrt2} + c\right) -
\frac{3c^2+2c+2}{16}$, that is $\frac{42}{2+\sqrt2} - 10 + 5c - 3c^2 >
0$. However, $\frac{42}{2+\sqrt2} > 10$, and, by~(C2), $5c > 3c^2$.

\medskip\noindent\textbf{(C5)} By~(A3), $a > \sqrt2$, so $a^2 > 2$.
By~(C4), $b < \frac12$, so $2b < 1$. By~(C2) and~(A1), $0 \le c < 1$,
so $c^2 < 1$. Thus, $-a^2 + 2b + c^2 < 0$. Therefore, $(-a^2 + 2b +
c^2)\sqrt{(a+c)^2 + 4b} < (-a^2 + 2b + c^2)(a+c)$, so if
\begin{equation}\label{e:ps6a}
a^3 - 4a + (a^2 - 4b)c - ac^2 - c^3 \ge (-a^2 + 2b + c^2)(a+c),
\end{equation}
then~(C5) holds. Thus, we will be proving~\eqref{e:ps6a}.

By~(A3), $b < (a-c)\sqrt2 - 2 =: b_1$. If~\eqref{e:ps6a} holds for $b
= b_1$, then replacing $b_1$ by $b$ in~\eqref{e:ps6a} will make the
left-hand side larger and the right-hand side smaller,
so~\eqref{e:ps6a} still holds. Therefore, it is enough to
prove~\eqref{e:ps6a} for $b = b_1$, that is
\begin{equation}\label{e:ps6b}
a^3 + (c - \sqrt2)a^2 + (-c^2 - 2\sqrt2 c)a + (-c^3 + 3\sqrt2 c^2 +
6c) \ge 0.
\end{equation}
The derivative of the left-hand side of~\eqref{e:ps6b} with respect to
$a$ is
\[
3a^2 + 2(c-\sqrt2)a + (-c^2-2\sqrt2 c) = (a+c)(3a -c - 2\sqrt2).
\]
Since $a \ge \sqrt2$, we have $3a - c - 2\sqrt2 \ge 3\sqrt2 - 1 -
2\sqrt2 > 0$. Therefore, the left-hand side of~\eqref{e:ps6b}
increases with $a$. Thus, it is enough to check \eqref{e:ps6b} for $a
= \sqrt2$:
\[
2\sqrt2 + 2(c-\sqrt2) + \sqrt2 (-c^2 - 2\sqrt2 c) + (-c^3 + 3\sqrt2
c^2 + 6c) = 4c + c^2(2\sqrt2 - c) > 0.
\]

\subsection{Existence of universal cones}\label{AppW}

Both $d$ and $b/d$ are roots of the equation $\lambda^2-(a-c)\lambda +
b = 0$, so in particular, $d+b/d=a-c$.

In order to prove~\eqref{eq:stable}, we may assume that $y'=1$ and
$|x'|\le d/b$. Then $x=1/b$ and $y=x'-\frac{\pm a-c}{b}$, and we have
to prove that
\begin{equation}\label{eq:st1}
1\le d\left|x'-\frac{\pm a-c}{b}\right|,\ \ \frac{d}{b}\le d\cdot
\frac1b,\ \ 1\le d\left|x'-\frac{\pm a-c}{b}\right|.
\end{equation}
The second inequality holds and the first and third are identical. We
have
\[
\left|x'-\frac{\pm a-c}{b}\right|\ge\frac{a-c}{b}=|x'|\ge
\frac{a-c}b-\frac{d}b=\frac{b}{db}=\frac1d,
\]
so~\eqref{eq:st1} holds.

Similarly, in order to prove~\eqref{eq:unstable}, we may assume that
$x=1$ and $|y|\le d$. Then $x'=\pm a-c+y$ and $y'=b$, and we have to
prove that
\begin{equation}\label{eq:un1}
b\le d|\pm a-c+y|,\ \ b\ge\frac{b}{d}\cdot d,\ \ |\pm a-c+y|\ge
\frac{b}{d}.
\end{equation}
Again, the second inequality holds and the first and third are
equivalent. We have
\[
|\pm a-c+y|\ge a-c-d=\frac{b}{d},
\]
so~\eqref{eq:un1} holds.

\subsection{$F^3(D)$ lies to the left of the line through $X$ and $M$}
\label{AppA}

There are several ways to find a condition which implies that $F^3(D)$
lies to the left of the line through $X$ and $M$. The simplest
condition is $y_3 - y_Q > 0$, where $Q = (x_Q, y_Q)$ denotes the point
on the line through $X$ and $M$ such that $x_Q = x_3$.
Calculation shows that $y_Q = sx_3 + y_M$, where
\[
s = \frac{2 b}{-a - c +\sqrt{4 b + (a + c)^2}}
\]
is the slope of the line through $X$ and $M$, and
\[
y_M = \frac{b}{1 + a - b + c} - \frac{2 b}{(1 + a - b + c) (-a - c + \sqrt{4
  b + (a + c)^2})}.
\]
Now the inequality $y_3 - y_Q > 0$ follows from~(C5).

\subsection{$B$ is above $M$}\label{AppB}

If $F^3(D)$ lies in the right half-plane, then $B \in [F^3(D),
F^2(D)]$ and by Appendix~\ref{AppA}, $B$ is above $M$. If $F^3(D)$
lies in the left half-plane, then $B \in [F(D), F^3(D)]$. Let $N =
(x_N, y_N)$ denote the point of intersection of $[X, M]$ and $[F(D),
F^3(D)]$. Then $B$ being above $M$ is equivalent to $x_N > 0$.
Calculation shows that $x_N = y_M + s_{13}x_1/(s_{13} - s)$, where
\[
s_{13} = \frac{b (3 a - c + \sqrt{4 b + (a + c)^2})}{-2 b + 4 a (a -
  c)}
\]
is the slope of the line through $F(D)$ and $F^3(D)$. Now the
inequality $x_N > 0$ follows from~(C4).

\subsection{$F^4(D)$ lies to the right of the line through $F^2(D)$
  and $F^3(D)$}\label{AppC}

A condition which implies that $F^4(D)$ lies to the right of the line
through $F^2(D)$ and $F^3(D)$ is $x_4 - x_P > 0$, where $P = (x_P,
y_P)$ denotes the point on the line through $F^2(D)$ and $F^3(D)$ such
that $y_P = y_4$. Calculation shows that $y_P = x_2 + (y_4 -
y_2)/s_{23}$, where
\[
s_{23} = \frac{2 b^2}{2 a^2 + a \left(3 b + 2 c - 2 \sqrt{4 b + (a +
    c)^2}\right) - b \left(c + \sqrt{4 b + (a + c)^2}\right)}
\]
is the slope of the line through $F^2(D)$ and $F^3(D)$. Now the
inequality $x_4 - x_P > 0$ follows from~(A2).

\subsection{Differences between the families $L_{a,b}$ and
  $F_{a,b,c}$}\label{AppD}

Condition $F^4_{a,b,c}(D) \in [X, M]$ is equivalent to the equality
$y_4 - y_X - s(x_4 - x_X) = 0$ (where $X = (x_X, y_X)$), and gives the
equation~(\ref{F1}).

For the assumption~(F2), note first that $Q$ has coordinates
\[
Q = (x_Q, y_Q) = \left(\frac{1 + a - b - c}{a^2 + (-1 + b)^2 - c^2},
-\frac{b (-1 + a + b + c)}{a^2 + (-1 + b)^2 - c^2}\right),
\]
and the stable manifold of $F_1\circ F_2$ of $Q$ has slope
\[
s_2 = \frac{2 b (a + c)}{a^2 - c^2 - \sqrt{-4 b^2 + (a^2 - 2 b -
    c^2)^2}}.
\]
Also,
\[
F(B) = (x_{B_1}, y_{B_1}) = \left(1 - \frac{b (3 a - c + \sqrt{4 b +
    (a + c)^2})(2 + a + c + \sqrt{4 b + (a + c)^2})}{2 (-2 b + 4 a (a
  - c))(1 + a - b + c)}, 0\right).
\]
The condition that $F(B)$ lies on the stable manifold of $F_1\circ
F_2$ of $Q$ is equivalent to the equality $y_Q - s_2(x_Q - x_{B_1}) =
0$ and gives the equation~(\ref{F2}).

\medskip
\noindent
Micha{\l} Misiurewicz\\
Department of Mathematical Sciences\\
Indiana University -- Purdue University Indianapolis\\
402 N.\ Blackford Street, Indianapolis, IN 46202\\
\texttt{mmisiure@math.iupui.edu}\\
\texttt{http://www.math.iupui.edu/}$\sim$\texttt{mmisiure}

\medskip
\noindent
Sonja \v Stimac\\
Department of Mathematics\\
Faculty of Science, University of Zagreb\\
Bijeni\v cka 30, 10\,000 Zagreb, Croatia\\
\texttt{sonja@math.hr}\\
\texttt{http://www.math.hr/}$\sim$\texttt{sonja}

\end{document}